\documentclass[12pt]{article}
\usepackage{amsfonts}
\usepackage{float}
\usepackage{mathrsfs}
\usepackage{amsmath,amscd,amsthm,mathrsfs,indentfirst,amssymb,cite}
\usepackage{upgreek}
\allowdisplaybreaks[4]

\begin{document}
\newtheorem{Theorem}{Theorem}[section]
\newtheorem{Lemma}[Theorem]{Lemma}
\newtheorem{Proposition}[Theorem]{Proposition}
\newtheorem{Example}[Theorem]{Example}
\newtheorem{Definition}[Theorem]{Definition}
\newtheorem{Corollary}[Theorem]{Corollary}

\title{On $\mathscr{H}C$-subgroups of a finite group\thanks{This manuscript was finished in August 2013.} \thanks{Research is supported by a NNSF grant of China (Grant 11371335)
and Research Fund for the Doctoral Program of Higher Education of China(Grant 20113402110036).}
\author{Lijun Huo, Xiaoyu Chen\thanks{Corresponding author.}, Wenbin Guo\\
\small Department of Mathematics, University of Science and
Technology of China, \\ \small Hefei, Anhui 230026, P. R. China \\
\small e-mails: ljhuo@mail.ustc.edu.cn, jelly@mail.ustc.edu.cn, wbguo@ustc.edu.cn}
}
\date{}
\maketitle

\begin{abstract}A subgroup $H$ of a finite group $G$ is said to be an $\mathscr{H}C$-subgroup of $G$ if there exists a
normal subgroup $T$ of $G$ such that $G=HT$ and $H^g \cap N_T(H)\leq H$ for all $g\in G$.
 In this paper, we investigate the structure of a finite group $G$ under the assumption that certain
subgroups of $G$ of arbitrary prime power order are $\mathscr{H}C$-subgroups of $G$.\end{abstract}
\renewcommand{\thefootnote}{\empty}
\section{Introduction}
Throughout this paper, all groups considered are finite. $G$ always denotes a group, $p$ denotes a prime, and $|G|_p$ denotes the order of Sylow $p$-subgroups of $G$. A class of groups $\mathfrak{F}$ is called a formation if $\mathfrak{F}$ is closed under taking homomorphic images and subdirect products. A formation $\mathfrak{F}$ is said to be saturated if $G\in \mathfrak{F}$ whenever $G/\Phi(G)\in \mathfrak{F}$. All
unexplained notation and terminology are standard, as in \cite{09,14,15}.\par
Recall that a subgroup $H$ of $G$ is said to be an $\mathscr{H}$-subgroup of $G$ if $H^g \cap N_G(H)\leq H$ for
all $g\in G$. This concept was introduced by Goldschmidt in \cite{30} and Bianchi et al. in \cite{01}.
It is easy to see that normal subgroups, Sylow subgroups and self-normalizing
subgroups of $G$ are all $\mathscr{H}$-subgroups of $G$.
Cs$\ddot{\rm{o}}$rg$\ddot{\rm{o}}$, Herzog \cite{02} and Asaad \cite{03} further investigated the influence of
$\mathscr{H}$-subgroups on the structure of a finite group.\par
Besides, Y. Wang \cite{33} introduced the concept of c-normal subgroups. A subgroup $H$ of $G$ is said to be c-normal in $G$ if there exists a normal subgroup $K$ of $G$ such that $G=HK$ and $H\cap K\leq H_G$, where $H_G$ is the largest normal subgroup of $G$ contained in $H$. The properties of c-normal subgroups have been studied by many authors, see for example, \cite{36,37,38,39}.\par
Recently, some attempts were made to give a generalization of both $c$-normal subgroups and $\mathscr{H}$-subgroups. In \cite{99}, M. Asaad et al. introduced the concept of weakly $\mathscr{H}$-subgroups: a subgroup $H$ of $G$ is called an $\mathscr{H}C$-subgroup of $G$ if there exists a
normal subgroup $T$ of $G$ such that $G=HT$ and $H\cap T$ is an $\mathscr{H}$-subgroup of $G$. Meanwhile, X. Wei and X. Guo \cite{05} introduced the concept of $\mathscr{H}C$-subgroups: a subgroup $H$ of $G$ is said to be an $\mathscr{H}C$-subgroup of $G$ if there exists a
normal subgroup $T$ of $G$ such that $G=HT$ and $H^g \cap N_T(H)\leq H$ for all $g\in G$. It is easy to see that every weakly $\mathscr{H}$-subgroup of $G$ is an $\mathscr{H}C$-subgroup of $G$.

In \cite{05}, the authors
gave some conditions on maximal subgroups or minimal subgroups of Sylow subgroups, which are sufficient to guarantee a group to be $p$-nilpotent or supersolvable.
In this paper, we continue to investigate the structure of a group $G$ under the assumption that certain
subgroups of $G$ of arbitrary prime power order are $\mathscr{H}C$-subgroups of $G$. New characterizations of some classes of finite groups are obtained.

\section{Preliminaries}

\begin{Lemma}
Suppose that $H$ is an $\mathscr{H}$-subgroup of $G$.

$(1)$ \textup{(\cite[Theorem 6(2)]{01})} If H is subnormal in G, then H is normal in G.

$(2)$ \textup{(\cite[Lemma 7(2)]{01})} If $H\leq K\leq G$, then $H$ is an $\mathscr{H}$-subgroup of $K$.

$(3)$ \textup{(\cite[Lemma 2(1)]{01})} If $N\leq H$ and $N\unlhd G$, then $H$ is an $\mathscr{H}$-subgroup of $G$ if and only if $H/N$ is an $\mathscr{H}$-subgroup of $G/N$.

$(4)$ \textup{(\cite[Theorem 6(3)]{01})} If $N\unlhd G$ and $N\leq N_G(H)$, then $N_G(HN)=N_G(H)$ and $HN$ is an $\mathscr{H}$-subgroup of $G$.
\end{Lemma}

\begin{Lemma}\label{2.2}Let $H$ and $K$ be subgroups of $G$, and $N\unlhd G$.

$(1)$ \textup{(\cite[Lemma 2.3(1)]{05})} If $H\leq K$ and $H$ is an $\mathscr{H}C$-subgroup of $G$, then $H$ is an $\mathscr{H}C$-subgroup of $K$.

$(2)$ \textup{(\cite[Lemma 2.3(2)]{05})} If $N\leq H$, then $H$ is
an $\mathscr{H}C$-subgroup of $G$ if and only if $H/N$ is an $\mathscr{H}C$-subgroup of $G/N$.

$(3)$ \textup{(\cite[Lemma 2.4]{05})} If $H$ is a $p$-group with $(p, |N|)=1$ and $H$ is an $\mathscr{H}C$-subgroup of $G$, then $HN$ is an
$\mathscr{H}C$-subgroup of $G$ and $HN/N$ is an $\mathscr{H}C$-subgroup of $G/N$.
\end{Lemma}

\begin{Lemma}\label{2.2a}$[\ref{05},~\rm{Theorem~3.7}]$
Let p be the smallest prime dividing $|G|$ and let P be
a Sylow $p$-subgroup of G. Then G is p-nilpotent if every maximal subgroup of P is an
$\mathscr{H}C$-subgroup of G.
\end{Lemma}

\begin{Lemma}\textup{\cite[Corollary B3]{30}}
Suppose that $S$ is a \textup{2}-subgroup of $G$ such that $S$ is an
$\mathscr{H}$-subgroup of G and $N_G(S)/C_G(S)$ is a \textup{2}-group. Then $S$ is a Sylow \textup{2}-subgroup of $S^G$.
\end{Lemma}

\begin{Lemma}\textup{\cite[Theorem 1]{10}}
Let P be a Sylow p-subgroup of G. Then the following two statements are true:\par
\textup{(1)} If $p$ is odd and every minimal subgroup of $P$ lies in $Z(N_G(P))$, then $G$ is $p$-nilpotent.\par
\textup{(2)} If $p=2$ and every cyclic subgroup of $P$ of order $2$ or $4$ is quasinormal in $N_G(P)$, then $G$ is $2$-nilpotent.\par
\end{Lemma}
Let $F^*(G)$ denote the generalized Fitting subgroup of $G$, that is, the largest
normal quasinilpotent subgroup of $G$. The following basic facts can be found in [\ref{12}, Chapter X].\par
\begin{Lemma}\label{2.7}
$(1)$ If $N$ is a normal subgroup of $G$, then $F^*(N)=N\cap F^*(G)$.

$(2)$ $F(G)\leq F^*(G)=F^*(F^*(G))$. If $F^*(G)$ is solvable, then $F^*(G)=F(G)$.

$(3)$ $C_G(F^*(G))\leq F(G)$.

$(4)$ If $G>1$, then $F^*(G)>1$. In fact, $F^*(G)/F(G)=soc(F(G)C_G(F(G))/F(G))$.
\end{Lemma}
\begin{Lemma}\textup{\cite[Lemma 2.5]{05}}
Let $K$ be a normal subgroup of $G$ and let $H$ be a normal subgroup of $K$. If $H$ is an $\mathscr{H}C$-subgroup of G, then $H$ is c-normal in $G$.
\end{Lemma}

\begin{Lemma}\label{2.4}
Let $H$ be a $p$-subgroup of G.
If $H$ is an $\mathscr{H}C$-subgroup of $G$ and $H$ is not an $\mathscr{H}$-subgroup of $G$, then G has a normal subgroup $M$ such that
$|G:M|=p$ and $G=HM$.
\end{Lemma}

\begin{proof}
By hypothesis, $G$ has a normal subgroup $T$ such that $G=HT$ and $H^g\cap N_T(H)\leq H$
for all $g\in G$. Since $H$ is not an $\mathscr{H}$-subgroup of $G$, we have that $T<G$. Hence $G/T$ is a $p$-group,
and so $G$ has a normal subgroup $M$ containing $T$ such that $|G:M|=p$ and $G=HM$.
\end{proof}

\begin{Lemma}\textup{\cite[Lemma 2.9]{23}}
Let $\mathfrak{F}$ be a saturated formation containing all supersolvable groups and let $G$ be a group with a normal subgroup $E$ such that $G/E\in \mathfrak{F}$. If $E$ is cyclic, then $G\in \mathfrak{F}$.
\end{Lemma}
\begin{Lemma}\label{2.8}
Let $\mathfrak{F}$ be a saturated formation containing all supersolvable groups. Suppose that M is a subgroup of G such that $|G:M|=p$, $F(G)\nleq M$ and $M\in \mathfrak{F}$. Then $G\in \mathfrak{F}$.
\end{Lemma}
\begin{proof}
If $\Phi(G)>1$, then it is easy to see that $G/\Phi(G)$
satisfies the hypothesis of the lemma, and so $G/\Phi(G)\in\mathfrak{F}$ by induction. This implies that $G\in \mathfrak{F}$.
We may, therefore, assume that $\Phi(G)=1$. Then $F(G)=N_1\times N_2\cdots \times N_t$,
where $N_i~(i=1,...,t)$ is a solvable minimal normal subgroup of $G$. Since $F(G)\nleq M$, there exists a solvable minimal normal subgroup $N_i$ of $G$ such that $N_i\nleq M$. Then clearly, $G=N_iM$, and so $N_i\cap M=1$. Therefore, $|N_i|=|G:M|=p$ and $G/N_i\cong M\in \mathfrak{F}$. It follows from Lemma 2.9 that $G\in \mathfrak{F}$.
\end{proof}

\begin{Lemma}\label{2.3}
Let $H$ be an $\mathscr{H}C$-subgroup of $G$.
If $L/\Phi(L)$ is a chief factor of $G$ and $H\leq L$, then $H$ is an $\mathscr{H}$-subgroup of $G$.
\end{Lemma}

\begin{proof}
By hypothesis, there exists a normal subgroup $T$ of $G$ such that $G=HT$ and $H^g\cap N_T(H)\leq H$
for all $g\in G$. Since $L/\Phi(L)$ is a chief factor of $G$, either $(L\cap T)\Phi(L)/\Phi(L)=1$
or $(N\cap T)\Phi(L)/\Phi(L)=L/\Phi(L)$. In the former case, $L=H(L\cap T)=H$. This implies that $H\unlhd G$, and so $H$ is an $\mathscr{H}$-subgroup of $G$.
In the latter case, $(L\cap T)\Phi(L)=L$, and so $T=G$. This also implies that
$H$ is an $\mathscr{H}$-subgroup of $G$.
\end{proof}

\begin{Lemma} \textup{\cite[Lemma 2.8]{34}}
Let $P$ be a normal
$p$-subgroup of $G$ contained in $Z_{\infty}(G)$. Then $O^p(G)\leq C_G(P)$.
\end{Lemma}

\begin{Lemma}\textup{\cite[Lemma 2.4]{35}}
Let $P$ be a $p$-group. If $\alpha$ is a $p'$-automorphism of $P$ which centralizes $\mathnormal{\Omega}_1(P)$, then $\alpha=1$ unless $P$ is a non-abelian $2$-group. If $[\alpha, \mathnormal{\Omega}_2(P)] = 1$, then $\alpha=1$ without restriction.
\end{Lemma}
\section{Main results}

\begin{Theorem}\label{3.1}
Let $p$ be the smallest prime divisor of $|G|$ and let $P$ be a Sylow $p$-subgroup of $G$. Suppose that $P$ is cyclic or $P$ has a subgroup $D$ with $1<|D|<|P|$ such that every subgroup $H$ of $P$ of order $|D|$ is an $\mathscr{H}C$-subgroup of $G$. When $p=2$ and $|P:D|>2$, suppose further that $H$ is an $\mathscr{H}C$-subgroup of $G$ if there exists $D_1\unlhd H\leq P$ such that $2|D_1|=|D|$ and $H/D_1$ is a cyclic group of order $4$. Then $G$ is $p$-nilpotent.
\end{Theorem}
\begin{proof}
Suppose that the result is false and let $G$ be a counterexample of minimal order. Then we proceed via the following steps.\par
\medskip
\noindent(1) $P$ is not cyclic and $|P:D|>p$.\par
If $P$ is cyclic, then by \cite[(10.1.9)]{08}, $G$ is $p$-nilpotent, a contradiction. Suppose that $|P:D|=p$. Then every maximal subgroup of $P$ is an $\mathscr{H}C$-subgroup of $G$. Hence by Lemma 2.3, $G$ is $p$-nilpotent, also a contradiction.\par
\medskip
\noindent(2) Every proper subgroup of $G$ containing $P$ is $p$-nilpotent.\par
Let $V$ be any proper subgroup of $G$ containing $P$. Then by Lemma 2.2(1), $V$ satisfies the hypothesis of the theorem. By the choice of $G$, $V$ is $p$-nilpotent. Thus (2) follows.\par
\medskip
\noindent(3) $O_{p'}(G)=1$.\par
If not, then by Lemma 2.2(3), $G/O_{p'}(G)$ satisfies the hypothesis of the theorem. By the choice of $G$, $G/O_{p'}(G)$ is $p$-nilpotent, and so $G$ is $p$-nilpotent, which is impossible.\par
\medskip
\noindent(4) $G$ is not a non-abelian simple group.\par
Assume that $G$ is a non-abelian simple group. Then by Feit-Thompson's Theorem, we have that $p=2$.
Let $H$ be a subgroup of $P$ of order $|D|$. Then clearly, $H$ is an $\mathscr{H}$-subgroup of $G$. Hence by Lemma 2.1(2), $H$ is an $\mathscr{H}$-subgroup of $P$, and thus $H\unlhd P$ by Lemma 2.1(1). It follows from (2) that $N_G(H)$ is $2$-nilpotent for $H\ntrianglelefteq G$, and so $N_G(H)/C_G(H)$ is a 2-group. By Lemma 2.4, $H$ is a Sylow 2-subgroup of $G$, a contradiction. Therefore, $G$ is not a non-abelian simple group.\par
\medskip
\noindent(5) $O_{p}(G)>1$, and every proper normal subgroup of $G$ is contained in $O_p(G)$.\par
Let $L$ be a proper normal subgroup of $G$. Then we only need to prove that $L$ is a $p$-group. By (3), $p\mid|L|$. If $|L|_p>|D|$, then $L$ satisfies the hypothesis of the theorem by Lemma 2.2(1). Hence $L$ is $p$-nilpotent due to the choice of $G$. It follows from (3) that $L$ is a $p$-group. Now consider that $|L|_p\leq |D|$. Then there exists a normal subgroup $K$ of $P$ such that $P\cap L\leq K$ and $|K|=p|D|$. This induces that $|LK|_p=|K|=p|D|$, and so $K$ is a Sylow $p$-subgroup of $LK$. If $LK=G$, then $|P|=|K|=p|D|$, which contradicts (1). Thus $LK<G$. By Lemma 2.2(1), $LK$ satisfies the hypothesis of the theorem. Then by the choice of $G$, $LK$ is $p$-nilpotent, and so is $L$. Hence $L$ is a $p$-group by (3), and consequently (5) holds.\par
\medskip
\noindent(6) Every $\mathscr{H}C$-subgroup of $G$ contained in $P$ is an $\mathscr{H}$-subgroup of $G$.\par
Let $V$ be any $\mathscr{H}C$-subgroup of $G$ contained in $P$. Then there exists a normal subgroup $T$ of $G$ such that $G=VT$ and $V^g \cap N_T(V)\leq V$ for all $g\in G$. By (5), since $G$ is not a $p$-group, we have that $T=G$. Therefore, $V$ is an $\mathscr{H}$-subgroup of $G$.\par
\medskip
\noindent(7) Let $N$ be a minimal normal subgroup of $G$ contained in $O_p(G)$. Then $|N|\leq |D|$ and $G/N$ is $p$-nilpotent.\par
If $|N|>|D|$, then there exists a subgroup $H$ of $N$ of order $|D|$ such that $H$ is an $\mathscr{H}$-subgroup of $G$ by (6). It follows from Lemma 2.1(1) that $H$ is normal in $G$, which is impossible. Hence $|N|\leq |D|$. First suppose that $|N|<|D|$. Then by (6) and Lemma 2.1(3), $G/N$ satisfies the hypothesis of the theorem. By the choice of $G$, $G/N$ is $p$-nilpotent.\par
Now consider that $|N|=|D|$. We claim that every cyclic subgroup of $P/N$ of order prime or 4 (when $p=2$) is normal in $N_G(P)/N$. Let $X/N$ be a subgroup of $P/N$ of order $p$. If $N\leq \Phi(X)$, then $X$ is cyclic, and so is $N$. This implies that $|N|=|D|=p$. Then by (6), every cyclic subgroup of $P$ of order $p$ or 4 (when $p=2$) is an $\mathscr{H}$-subgroup of $G$. By Lemmas 2.1(1) and 2.1(2), every cyclic subgroup of $P$ of order $p$ or 4 (when $p=2$) is normal in $N_G(P)$. Since $p$ is the smallest prime divisor of $|G|$, every minimal subgroup of $P$ lies in $Z(N_G(P))$. Hence by Lemma 2.5, $G$ is $p$-nilpotent, a contradiction. Thus $N\nleq \Phi(X)$, and so $X$ has a maximal subgroup $S$ such that $X=SN$. Since $|S|=|N|=|D|$, $S$ is an $\mathscr{H}$-subgroup of $G$ by (6). By Lemmas 2.1(1) and 2.1(2), $S\unlhd N_G(P)$, and thus $X/N=SN/N\unlhd N_G(P)/N$. This shows that the claim holds when $p$ is odd. Consider that $p=2$. Then by (1), $|P:D|>2$. Let $Y/N$ be a cyclic subgroup of $P/N$ of order $4$. If $N\leq \Phi(Y)$, then $Y$ is cyclic. This implies that $|N|=|D|=2$, a contradiction. Thus $N\nleq \Phi(Y)$, and so $Y$ has a maximal subgroup $U$ such that $Y=UN$. Clearly, $|U|=2|D|$. Since $U/U\cap N\cong Y/N$ is a cyclic group of order 4, by hypothesis and (6), $U$ is an $\mathscr{H}$-subgroup of $G$. A similar discussion as above shows that $Y/N=UN/N\unlhd N_{G}(P)/N$. Hence the claim holds when $p=2$. Since $p$ is the smallest prime divisor of $|G|$, every minimal subgroup of $P/N$ lies in $Z(N_G(P)/N)$. Therefore, $G/N$ is $p$-nilpotent by Lemma 2.5.\par
\medskip
\noindent(8) Final contradiction.\par
By (7), $G/N$ is $p$-nilpotent. Then $G$ has a normal subgroup $M$ of $G$ such that $|G:M|=p$. By (1), $|M|_p>|D|$. Then by (6) and Lemma 2.1(2), $M$ satisfies the hypothesis of the theorem. Hence $M$ is $p$-nilpotent due to the choice of $G$, and so $G$ is $p$-nilpotent. The final contradiction completes the proof.
\end{proof}

The following corollary can be deduced immediately from Lemma 2.2(3) and Theorem 3.1.\par
\begin{Corollary}
Suppose that every
%Sylow subgroup $P$ of $E$ is cyclic or every
noncyclic Sylow subgroup $P$ $($if exists$)$ of $G$ has a subgroup $D$ such that
$1<|D|<|P|$ and every subgroup $H$ of $P$ of order $|D|$
is an $\mathscr{H}C$-subgroup of $G$. When $P$ is a Sylow $2$-subgroup of $G$ and $|P:D|>2$, suppose further that $H$ is an $\mathscr{H}C$-subgroup of $G$ if there exists $D_1\unlhd H\leq P$ such that $2|D_1|=|D|$ and $H/D_1$ is a cyclic group of order $4$. Then $G$ has a Sylow tower of
supersolvable type.
\end{Corollary}

\begin{Theorem}\label{3.5}
Let $\mathfrak{F}$ be a saturated formation containing all supersolvable groups and let $G$ be a group with a normal subgroup $E$ such that $G/E\in \mathfrak{F}$.
Suppose that every
noncyclic Sylow subgroup $P$ $($if exists$)$ of $F^*(E)$ has a subgroup $D$ such that
$1<|D|<|P|$ and every subgroup $H$ of $P$ of order $|D|$
is an $\mathscr{H}C$-subgroup of $G$. When $P$ is a Sylow $2$-subgroup of $F^*(E)$ and $|P:D|>2$, suppose further that $H$ is an $\mathscr{H}C$-subgroup of $G$ if there exists $D_1\unlhd H\leq P$ such that $2|D_1|=|D|$ and $H/D_1$ is a cyclic group of order $4$. Then $G\in \mathfrak{F}$.
\end{Theorem}

\begin{proof}
Suppose that the result is false and let $(G,E)$ be a counterexample such that $|G|+|E|$ is minimal. Then we proceed via the following steps.\par
\medskip
(1) $F^*(E)=F(E)$.\par
By Lemma 2.2(1) and Corollary 3.2, $F^*(E)$ has a Sylow tower of supersolvable type, and so $F^*(E)$ is solvable. It follows from Lemma 2.6(2) that $F^*(E)=F(E)$.\par
\medskip
(2) There exist a noncyclic Sylow $p$-subgroup $P$ of $F(E)$ and a subgroup $H$ of $P$ of order $|D|$ or 2$|D|$ (when $p=2$, and there exists $D_1\unlhd H\leq P$ such that $2|D_1|=|D|$ and $H/D_1$ is a cyclic group of order $4$) such that $|P:D|>p$ and $H$ is not an $\mathscr{H}$-subgroup of $G$.\par
Suppose that for every prime divisor $p$ of $|F(E)|$ and every noncyclic Sylow $p$-subgroup $P$ of $F(E)$, either $|P:D|=p$ or all subgroups $H$ of $P$ of order $|D|$ or 2$|D|$ (when $p=2$, $|P:D|>2$ and there exists $D_1\unlhd H\leq P$ such that $2|D_1|=|D|$ and $H/D_1$ is a cyclic group of order $4$) are $\mathscr{H}$-subgroups of $G$. In the former case, by Lemma 2.7, all subgroups $H$ of $P$ of order $|D|$ are c-normal in $G$. In the latter case, by Lemma 2.1(1), all subgroups $H$ of $P$ of order $|D|$ or 2$|D|$ (when $p=2$, $|P:D|>2$ and there exists $D_1\unlhd H\leq P$ such that $2|D_1|=|D|$ and $H/D_1$ is a cyclic group of order $4$) are normal in $G$. Hence by \cite[Theorem 1.4]{32}, $G\in \mathfrak{F}$, a contradiction. Thus (2) holds.\par
\medskip
(3) Final contradiction.\par
By hypothesis and (2), $H$ is an $\mathscr{H}C$-subgroup of $G$ and $H$ is not an $\mathscr{H}$-subgroup of $G$. Hence $G$ has a normal subgroup $M$ such that $|G:M|=p$ and $G=HM$ by Lemma 2.8. Since $G=HM=EM$, we have that $M/E\cap M\cong G/E\in \mathfrak{F}$. By Lemma 2.6(1), $F^*(E\cap M)=F^*(E)\cap M=F(E)\cap M$. Note that $|F(E):F(E)\cap M|=|G:M|=p$ and $|F(E)\cap M|_p>|D|$ by (2). Then clearly, $(M, E\cap M)$ satisfies the hypothesis of the theorem by Lemma 2.2(1). By the choice of $(G,E)$, $M\in \mathfrak{F}$. It follows from Lemma 2.10 that $G\in \mathfrak{F}$. The final contradiction completes the proof.
\end{proof}

\begin{Theorem}\label{3.4}
Let $\mathfrak{F}$ be a saturated formation containing all supersolvable groups and let $G$ be a group with a normal subgroup $E$ such that $G/E\in \mathfrak{F}$.
Suppose that every
noncyclic Sylow subgroup $P$ $($if exists$)$ of $E$ has a subgroup $D$ such that
$1<|D|<|P|$ and every subgroup $H$ of $P$ of order $|D|$
is an $\mathscr{H}C$-subgroup of $G$. When $P$ is a Sylow $2$-subgroup of $E$ and $|P:D|>2$, suppose further that $H$ is an $\mathscr{H}C$-subgroup of $G$ if there exists $D_1\unlhd H\leq P$ such that $2|D_1|=|D|$ and $H/D_1$ is a cyclic group of order $4$. Then $G\in \mathfrak{F}$.
\end{Theorem}

\begin{proof}
Suppose that the result is false and let $(G,E)$ be a counterexample such that $|G|+|E|$ is minimal. By Lemma 2.2(1) and Corollary 3.2, we see that $E$ has a Sylow tower of
supersolvable type.
Without loss of generality, let $p$ be the largest prime divisor of $|E|$.
Then $P\unlhd G$. By Lemma
\ref{2.2}(3), $(G/P,E/P)$ satisfies the hypothesis of the theorem.
Then the choice of $(G,E)$ implies that $G/P\in \mathfrak{F}$. Hence $(G,P)$ satisfies the hypothesis of Theorem 3.3, and so $G\in \mathfrak{F}$.
\end{proof}

\begin{Theorem}\label{3.7}
Let E be a normal subgroup of G such that $G/E$ is
nilpotent. Suppose that every minimal subgroup of $E$ is contained in $Z_\infty(G)$, and
every cyclic subgroup of E of order $4$ is an $\mathscr{H}C$-subgroup of $G$. Then G is nilpotent.
\end{Theorem}

\begin{proof}
Assume that the result is false and let $(G,E)$ be a counterexample such that
$|G|+|E|$ is minimal. Then we prove the theorem via the following steps.\medskip

(1) $G$ is a minimal nonnilpotent group, that is, $G=P\rtimes Q$, where $P$ is a normal Sylow $p$-subgroup of $G$
and $Q$ is a nonnormal cyclic Sylow $q$-subgroup of $G$ for some prime $q\neq p$;
$P/\Phi(P)$ is a chief factor of $G$; $exp(P)=p$ when $p > 2$ and $exp(P)$ is at most 4 when $p=2$.\par
Let $K$ be any proper subgroup of $G$. Then $K/E\cap K \cong EK/E\leq G/E$
is nilpotent, and every minimal subgroup of $E\cap K$ is contained
in $Z_\infty(G)\cap K \leq
Z_\infty(K)$. By hypothesis, every cyclic subgroup of $E\cap K$ of
order 4 is an $\mathscr{H}C$-subgroup of $G$. Thus by Lemma \ref{2.2}(1), every
cyclic subgroup of $E\cap K$ of order 4 is an $\mathscr{H}C$-subgroup of $K$.
Hence $(K,E\cap K)$ satisfies the hypothesis of the theorem.
Then the choice of $(G,E)$ implies that $K$ is nilpotent. Hence $G$ is a minimal nonnilpotent group, and so (1) holds by [\ref{09}, Chapter III, Satz 5.2].\par\medskip
(2) $P\leq E$.\par
If not, then $P\cap E<P$, and so $(P\cap E)Q<G$. By (1), $(P\cap E)Q$ is nilpotent. This implies that $Q\unlhd (P\cap E)Q$. Since $G/P\cap E\lesssim G/P\times G/E$ is nilpotent, $(P\cap E)Q\unlhd G$, and thus $Q\unlhd G$, a contradiction.\par
\medskip
(3) Final contradiction.\par
If $exp(P)=p$, then $P\leq Z_\infty(G)$, and so $G$ is nilpotent, which is impossible. Hence we may assume that $p=2$ and $exp(P)=4$. Then by Lemma 2.11, every cyclic subgroup of $P$ of order $4$ is an $\mathscr{H}$-subgroup of $G$, and so every cyclic subgroup of $P$ of order $4$ is normal in $G$ by Lemma 2.1(1). Take an element $x\in P\backslash \Phi(P)$. Since $P/\Phi(P)$ is a chief factor of $G$, $P=\langle x\rangle^G\Phi(P)=\langle x\rangle^G$. If $x$ is of order 2, then $P=\langle x\rangle^G\leq Z_\infty(G)$, a contradiction. Now assume that $x$ is of order 4. Then $\langle x\rangle\unlhd G$, and so $P=\langle x\rangle$ is cyclic. By \cite[(10.1.9)]{08}, $G$ is 2-nilpotent, and so $Q\unlhd G$. This is the final contradiction.\end{proof}

\begin{Theorem}\label{3.9}
Let E be a normal subgroup of G such that $G/E$ is
nilpotent. Suppose that every minimal subgroup of $F^*(E)$ is contained in $Z_\infty(G)$, and
every cyclic subgroup of $F^*(E)$ of order $4$ is an $\mathscr{H}C$-subgroup of $G$. Then G is nilpotent.
\end{Theorem}

\begin{proof}
Assume that the result is false and let $(G,E)$ be a counterexample such that
$|G|+|E|$ is minimal. Then we prove the theorem via the following steps.\par\medskip
(1) Every proper normal subgroup of $G$ is nilpotent.\par
Let $K$ be any proper normal subgroup of $G$. Then $K/E\cap K \cong EK/E\leq G/E$
is nilpotent. By Lemma 2.6(1), $F^*(E\cap K)=F^*(E)\cap K$. Hence by Lemma 2.2(1), $(K,E\cap K)$ satisfies the hypothesis of the theorem.
The the choice of $(G,E)$ implies that $K$ is nilpotent.\par\medskip
(2) $E=G=\gamma_{\infty}(G)$ and $F^*(G)=F(G)<G$, where $\gamma_{\infty}(G)$ is the nilpotent residual of $G$.\par
If $E<G$, then $E$ is nilpotent by (1), and so $F^*(E)=F(E)=E$. By
Theorem \ref{3.7}, $G$ is nilpotent, a contradiction. Thus $E=G$. Now suppose that $F^*(G)=G$. Then by Theorem \ref{3.7} again, $G$ is nilpotent, which is impossible. Hence $F^*(G)<G$, and $F^*(G)=F(G)$ by (1). If $\gamma_{\infty}(G)<G$, then by (1), $\gamma_{\infty}(G)\leq F(G)$, and so $G/F(G)$ is nilpotent. It follows from Theorem \ref{3.7} that $G$ is nilpotent, a contradiction. Thus $\gamma_{\infty}(G)=G$.\par\medskip
(3) every cyclic subgroup of $F(G)$ of order $4$ is contained in $Z(G)$.\par
By hypothesis and (2), every cyclic subgroup $H$ of $F(G)$ of order $4$ is an $\mathscr{H}C$-subgroup of $G$. Then there exists a normal subgroup $T$ of $G$ such that $G=HT$ and $H^g \cap N_T(H)\leq H$ for all $g\in G$. If $T<G$, then $T\leq F(G)$ by (1), and thereby $F(G)=G$, a contradiction. Hence $T=G$, and so $H$ is an $\mathscr{H}$-subgroup of $G$. By Lemma 2.1(1), $H\unlhd G$. This implies that $G/C_G(H)$ is abelian. Then by (2), $C_G(H)=\gamma_{\infty}(G)=G$, and so $H\leq Z(G)$. Thus (3) holds.\par\medskip
(4) Final contradiction.\par
Let $p$ be any prime divisor of $|F(G)|$ and let $P$ be the Sylow $p$-subgroup of $F(G)$. Then $P\unlhd G$. If $p$ is odd, then by hypothesis, $\Omega_1(P)\leq Z_\infty(G)$. It follows from Lemma 2.12 that $O^p(G)\leq C_G(\Omega_1(P))$, and so $O^p(G)\leq C_G(P)$ by Lemma 2.13. Then by (2), $C_G(P)=\gamma_{\infty}(G)=G$. Now consider that $p=2$. Then by hypothesis and (3), $\Omega_2(P)\leq Z_\infty(G)$. A similar discussion as above also shows that $C_G(P)=G$. Therefore, we have that $C_G(F(G))=G$, which contradicts the fact that $C_G(F(G))\leq F(G)$ by (2) and Lemma 2.6(3). The proof is thus completed.\end{proof}

\end{document}